\documentclass[a4paper,11pt]{article}
\title{On elliptic curves with $p$-isogenies over quadratic fields}
\author{Philippe Michaud-Jacobs}

\makeatletter
\newcommand\notsotiny{\@setfontsize\notsotiny\@vipt\@viipt}
\makeatother
\usepackage{amsmath} 
\usepackage{amssymb} 
\usepackage{mathtools}
\usepackage{framed}
\usepackage{pifont} 
\usepackage{enumerate} 
\usepackage{wasysym} 
\usepackage{commath}
\usepackage{graphicx}
\usepackage{framed}
\usepackage{amsmath}
\usepackage{amssymb}
\usepackage{amsfonts}
\usepackage{mathrsfs}
\usepackage{mathtools}
\usepackage{url}
\usepackage{array}
\usepackage{amsthm}
\usepackage[english]{babel}
\usepackage[utf8]{inputenc}
\newtheorem{theorem}{Theorem}[section]
\newtheorem*{theorem*}{Theorem}

\newtheorem {corollary}[theorem]{Corollary}
\newtheorem {lemma}[theorem]{Lemma}
\newtheorem {proposition}[theorem]{Proposition}
\theoremstyle{definition}

\theoremstyle{remark}
\newtheorem{remark}[theorem]{Remark}
\newtheorem*{remark*}{Remark}
\usepackage{yfonts}
\usepackage{amsmath,amssymb,tikz-cd}
\usepackage{pdflscape}

\usepackage{etoolbox}
\apptocmd{\sloppy}{\hbadness 10000\relax}{}{}

\usepackage[numbers]{natbib}
\usepackage{cite}
\usepackage{mathtools, nccmath}

\newcommand{\Mod}[1]{\ (\mathrm{mod}\ #1)}

\usepackage[section]{placeins}

\providecommand{\Q}{\mathbb{Q}}
\providecommand{\Z}{\mathbb{Z}}

\newcommand{\Gal}{\mathrm{Gal}}
\renewcommand{\arraystretch}{2.0}

\usepackage{longtable}
\usepackage{hyperref}

\usepackage{multirow}

\newcommand{\Addresses}{{
  \bigskip
  \footnotesize

 \textsc{Mathematics Institute, University of Warwick, CV4 7AL, United Kingdom}\par\nopagebreak
  \textit{E-mail address}: \texttt{p.rodgers@warwick.ac.uk}
}}

\let\svthefootnote\thefootnote
\newcommand\freefootnote[1]{%
  \let\thefootnote\relax%
  \footnotetext{#1}%
  \let\thefootnote\svthefootnote%
}

\date{\vspace{-3ex}}

\begin{document}

\maketitle

\begin{abstract}

Let $K$ be a number field. For which primes $p$ does there exist an elliptic curve $E / K$ admitting a $K$-rational $p$-isogeny? Although we have an answer to this question over the rationals, extending this to other number fields is a fundamental open problem in number theory. In this paper, we study this question in the case that $K$ is a quadratic field, subject to the assumption that $E$ is semistable at the primes of $K$ above $p$. We prove results both for families of quadratic fields and for specific quadratic fields.
\end{abstract}


\section{Introduction}


One of the most important aspects of the study of elliptic curves is the investigation of the maps between them, and in particular their isogenies. Isogenies of prime degree are perhaps the most intriguing: a complete understanding would provide much insight into the arithmetic of elliptic curves, yet we still cannot answer some of the most basic questions about them. In this paper, we will investigate isogenies of prime degree over quadratic fields. \freefootnote{\emph{Date}: \date{\today}.}
\freefootnote{\emph{Keywords}: Elliptic curve, isogeny, irreducibility, Galois representation, quadratic field, modular curve.}
\freefootnote{\emph{MSC2020}: 11F80, 11G05, 11G18.}
\freefootnote{The author is supported by an EPSRC studentship and has previously used the name Philippe Michaud-Rodgers.}
Given an elliptic curve $E$ defined over a number field $K$, and a prime $p$, we say that $E$ \emph{admits a $K$-rational $p$-isogeny} if it admits an isogeny, $\varphi$, of degree $p$, satisfying $\varphi^\sigma = \varphi$ for any $\sigma \in \mathrm{Gal}(\overline{K} / K)$. Equivalent formulations are that $E$ has a $K$-rational subgroup of order $p$, or that the mod~$p$ Galois representation of $E$ is reducible. We simply call an isogeny \emph{rational} if it is $\Q$-rational. The key question we would like to answer is the following: given a number field $K$, for which primes $p$ does there exist an elliptic curve $E / K$ admitting a $K$-rational $p$-isogeny?
Thanks to the work of Mazur, we have a complete answer to this question over the rationals. 
\begin{theorem*}[Mazur, {\citep[Theorem~1]{ratisog}}]\label{maz} Let $p$ be a prime such that there exists an elliptic curve $E / \Q$ that admits a rational $p$-isogeny. Then \[ p  \in \{2,3,5,7,11,13,17,19,37,43, 67, 163 \}. \]
\end{theorem*}

Although this theorem was proven over 40 years ago, it has not been possible to obtain an analogous result for even a single other number field. Perhaps the most likely candidate for a similar result is a quadratic field of small discriminant. Recent work \citep[p.~5]{Banwait} has shown that this is possible assuming the generalised Riemann hypothesis, although removing this assumption seems to be out of reach at this time.

Apart from the intrinsic interest of studying isogenies of elliptic curves, perhaps one of the most spectacular consequences of Mazur's theorem is the role it plays in the proof of Fermat's Last Theorem. More generally, in the `modular approach' to studying Diophantine equations, one associates a Frey elliptic curve to a putative solution of a Diophantine equation, and applies Ribet's level-lowering theorem \citep[Theorem~1.1]{rib} to relate this Frey curve to a modular form. A key hypothesis in applying Ribet's theorem at a given prime $p$ is the non-existence of a rational $p$-isogeny. 

More recently, the modular approach has been applied over various number fields, most commonly over real quadratic fields. See \citep{realquad}, \citep{AnniSiksek}, and \citep{shifted} for a sample of papers that do this. In these examples, the Frey curve one constructs is defined over a number field, $K$, and in order to apply an analogue of Ribet's level-lowering theorem \citep[Theorem~7]{asym}, it is again necessary, for a given prime $p$, to rule out the existence of a $K$-rational $p$-isogeny. Since there is no analogue of Mazur's theorem over number fields, various methods have been used to achieve this. A further assumption in this analogue of Ribet's theorem is that the elliptic curve one is working with should be semistable at all primes of $K$ above p, which one may view as a natural condition in its own right. With the assumption of semistability at the primes of $K$ above $p$, it is possible to obtain results akin to Mazur's theorem, both for families of quadratic fields and for specific quadratic fields. Our main result is the following.

\begingroup
\renewcommand\thetheorem{1}
\begin{theorem}\label{Thm1}
Let $K$ be a real quadratic field and let $\epsilon$ be a fundamental unit of $K$. Let $n$ be the exponent of the class group of $K$ and assume $n \leq 3$. Let $p$ be a prime such that there exists an elliptic curve $E / K$ which admits a $K$-rational $p$-isogeny and is semistable at all primes of $K$ above $p$. Then  either 
\begin{itemize}
\item $p$ ramifies in $K$; or
\item $p \in \{2,3,5,7,11,13,17,19,37 \}$; or
\item $p$ splits in $K$ and $p \mid \mathrm{Norm}_{K / \Q}(\epsilon^{12}-1)$.
\end{itemize}
\end{theorem}
\endgroup

Although this theorem only considers the case $n \leq 3$, where $n$ is the exponent of the class group of $K$, we may obtain similar results for larger values of $n$. For example, in Section 5 we consider the case $n = 100$. We also note that the list of primes appearing in this theorem is the smallest possible: for each $p \in \{2,3,5,7,11,13,17,19,37 \}$ and $n \leq 3$ there exists a real quadratic field $K$ with class group exponent $n$ and an elliptic curve $E / K$ which has a $K$-rational $p$-isogeny and is semistable at all primes of $K$ above $p$ (see Remark \ref{conv_rem}).

If we work over a fixed quadratic field, which is not imaginary of class number $1$, then we can obtain more precise results. The following theorem considers certain `small' quadratic fields, both real and imaginary.

\begingroup
\renewcommand\thetheorem{2}
\begin{theorem}\label{Thm2} Let $K = \Q(\sqrt{d})$ with $d \in \{-5,2,3,5,6,7\}$. Let $p$ be a prime. There exists an elliptic curve $E / K$ which admits a $K$-rational $p$-isogeny and is semistable at all primes of $K$ above $p$ if and only if $p \in \{2,3,5,7,13,37\}$ or the pair $(d,p)$ appears in Table \ref{Tab2} below.
\begingroup 
\renewcommand*{\arraystretch}{1.5}
\begin{table}[ht!]
\begin{center}
\begin{tabular}{ |c|c|c|c|c|c|c|c|c } 
 \hline
 $d$ & $-5$ & $2$ & $3$ & $5$ & $6$ & $7$ \\
 \hline
 $p$ & $43$ & $11,19$  & $17,19$ & $17$ & $11,17$ & $11,17$ \\
 \hline
\end{tabular}
\caption{\label{Tab2} Remaining primes.}
\end{center}
\end{table} 
\endgroup
\end{theorem}
\endgroup
We highlight the fact that this is an `if and only if' statement. It is also possible to produce similar results for quadratic fields with large class group exponent. As an example, in Section 5.3 we consider a quadratic field with class group $\Z / 122 \Z$.

We now outline the rest of the paper. In Section 2, we analyse the situation over the rationals, and prove a result analogous to Theorem \ref{Thm2}. This result is a corollary to  Mazur's theorem stated above. In Section 3, we study the mod $p$ Galois representation of an elliptic curve with a $p$-isogeny, and we introduce the notions of isogeny characters and isogeny signatures. Next, in Section 4, by studying the ramification of these isogeny characters and by investigating certain properties of the modular curve $X_0(p)$, we see how the existence of an elliptic curve with a $p$-isogeny places stringent conditions on the prime $p$. This provides us with a method for ruling out the existence of such primes. In Section 5, we apply this method, combined with a study of quadratic points on modular curves, to prove Theorems \ref{Thm1} and \ref{Thm2}. We also consider further examples.

The \texttt{Magma} \citep{magma} code used to support the computations in this paper can be found at: 

\vspace{3pt}

 \url{https://warwick.ac.uk/fac/sci/maths/people/staff/michaud/c/}

\bigskip

I am grateful to my supervisors, Samir Siksek and Damiano Testa, for all their help and support. I would also like to thank Barinder Banwait for very helpful correspondence. Finally, I am grateful to the anonymous referee for many helpful suggestions that have improved the exposition of the paper.


\section{Elliptic curves with rational \texorpdfstring{$p$}.-isogenies}


We start with a short analysis of the situation over the rationals. Let $E / \Q$ be an elliptic curve and let $p$ be prime for which $E$ admits a rational $p$-isogeny. We will denote the kernel of this isogeny by $V_p$, which is a rational cyclic subgroup of order $p$. The pair $(E,V_p)$ then gives rise to a non-cuspidal point $x \in X_0(p)(\Q)$. The study of the modular curve $X_0(p)$, and in particular the Eisenstein quotient of its Jacobian, allowed Mazur to prove his celebrated result \citep[Theorem~1]{ratisog} (stated in the introduction), which classifies the primes $p$ for which $X_0(p)$ has non-cuspidal rational points. This result allows us to obtain an analogue of Theorem \ref{Thm2} quite easily. 

\begin{corollary}[Corollary to Mazur's theorem on isogenies]\label{ratcorol} There exists an elliptic curve $E / \Q$ which admits a rational $p$-isogeny and is semistable at $p$ if and only if \[ p \in \{2,3,5,7,13,37 \}. \]
\end{corollary}

\begin{proof} Suppose first that $E/ \Q$ is an elliptic curve which admits a rational $p$-isogeny and is semistable at $p$. By Theorem \ref{maz}, it will suffice to rule out the primes \[ p \in \{11,17,19,43,67,163 \}. \] For each of these values of $p$, the modular curve $X_0(p)$ has only finitely many non-cuspidal rational points and we let $x \in X_0(p)(\Q)$ denote one of these. Using \texttt{Magma}'s small modular curve package, we can write down an elliptic curve $F / \Q$ with a rational subgroup $W_p$ of order $p$ such that the pair $(F,W_p)$ gives rise to the point $x$. In each case, the curve $F$ (we have chosen) has additive potentially good reduction at $p$ (so $F$ is not semistable at $p$) and its $j$-invariant is not equal to $0$ or $1728$. We compute that $0 < v_p(\Delta(F))  < 6$ in each case. In particular, $F$ is minimal at $p$.

However, this alone is not enough to rule out the prime $p$. The pair $(F,W_p)$ is one representative for the point $x \in X_0(p)(\Q)$, and it is possible that a different representative is semistable at $p$. Suppose $(\hat{F},\hat{W}_p)$ also represents the point $x \in X_0(p)(\Q)$ for an elliptic curve $\hat{F} / \Q$ with a rational subgroup of order $p$. We note that $j(F) = j(\hat{F})$, so $\hat{F}$ also has potentially good reduction at $p$. The curves $F$ and $\hat{F}$ are isomorphic (over $\overline{\Q}$), and since $j(F) = j(\hat{F}) \notin \{0,1728\}$, the curves are quadratic twists of each other (up to isomorphism over $\Q$) by some squarefree $d \in \Z$, and so we may replace $\hat{F}$ by $F_d$, where $F_d$ denotes the quadratic twist of $F$ by $d$. Since $\Delta(F_d) = d^6 \cdot \Delta(F)$, we see that \[ v_p(\Delta(F_d)) = v_p(\Delta(F)) + 6v_p(d). \] It follows that $0 < v_p(\Delta(F_d)) < 12$, so $F_d$ is minimal at $p$ and $F_d$ does not have good reduction at $p$. So $F_d$ must have additive reduction at $p$.

For the converse, it suffices to find elliptic curves which have a rational $p$-isogeny and are semistable at $p$ for $p \in \{ 2,3,5,7,13,37\}$. Table \ref{TabRat} gives an example of such an elliptic curve in each case. We have chosen an elliptic curve of minimal conductor in each case.
\end{proof}

\begingroup 
\renewcommand*{\arraystretch}{1.5}
\begin{table}[ht!]
\begin{center}
\begin{tabular}{ |c|c|c|c|c|c|c|c| } 
 \hline
 $p$ & $2$ & $3$ & $5$ & $7$ & $13$ & $37$ \\
 \hline
 $E$ & 14a1 & 14a1 & 11a1 & 26b1 & 147b1 & 1225e1 \\
 \hline
 $N(E)$ & $2 \cdot 7$ & $2 \cdot 7$ & $11$ & $2 \cdot 13$ & $3 \cdot 7^2$ & $5^2 \cdot 7^2$ \\
 \hline
\end{tabular}
\caption{\label{TabRat} Elliptic curves for the proof of Corollary \ref{ratcorol}. We have used Cremona's labelling and $N(E)$ denotes the conductor of $E$.}
\end{center}
\end{table} 
\endgroup


\section{Isogeny characters and isogeny signatures}


We will now shift our attention to quadratic fields. In this section we will introduce two key concepts: \emph{isogeny characters} and \emph{isogeny signatures}. We will define these concepts in relation to our set-up, but they can be defined more generally for elliptic curves with $p$-isogenies over arbitrary number fields (see \citep{B-D}).

For the remainder of the paper, we will let $p$ be a prime and let $E / K$ be an elliptic curve over a quadratic field $K$ such that:
\begin{itemize} 
\item $E$ admits a $K$-rational $p$-isogeny; and
\item $E$ is semistable at all primes $\mathfrak{p} \mid p.$
\end{itemize}
In this section we will assume that:
\begin{itemize}
\item $p \geq 17$; and 
\item $p$ is unramified in $K$.
\end{itemize}
We denote by $\varphi$ this $K$-rational $p$-isogeny, and we write $V_p$ for its kernel throughout. The group $V_p$ is a $K$-rational cyclic subgroup of $E[p]$ of order $p$. Write $G_K = \mathrm{Gal}(\overline{K} / K)$. The mod $p$ Galois representation of $E$, $\overline{\rho}_{E,p} : G_K \rightarrow \mathrm{GL}_2(\mathbb{F}_p)$, is reducible, and we have that
\[\overline{\rho}_{E,p} \sim 
 \mbox{\Large%
    $ \begin{psmallmatrix} \lambda \; \; & *  \\[4pt] 0 \; \; & \chi_p \lambda^{-1} \end{psmallmatrix}$},
\]
where $\chi_p$ denotes the mod $p$ cyclotomic character. We call $\lambda : G_K \rightarrow \mathbb{F}_p^\times$ the \textbf{isogeny character} of $(E,V_p)$. This character gives the action of $G_K$ on $V_p$, and we can choose $R \in V_p$ such that for all $\sigma \in G_K$, \[ R^\sigma = \lambda(\sigma)R.\] 

Throughout, we will let $\tau$ denote the generator of $\Gal(K / \Q)$. By choosing an automorphism of $\overline{K}$ extending $\tau$, we may also view $\tau$ as an element of $\Gal(\overline{K} / \Q)$. The following lemma describes the isogeny characters of $(E^\tau, V_p^\tau)$ and $(E/V_p, E[p] /V_p)$. 

\begin{lemma}\label{sig1} Let $\lambda$ be the isogeny character of $(E,V_p)$. \begin{enumerate}[(i)]
\item The isogeny character of $(E^\tau,V_p^\tau)$ is $\lambda^{\tau}$, defined by $\lambda^{\tau}(\sigma) = \lambda(\tau \sigma \tau^{-1})$ for $\sigma \in G_K$.
\item The isogeny character of $(E/V_p, E[p]/V_p)$ is $\chi_p \lambda^{-1}$.
\end{enumerate}
\end{lemma}

\begin{proof} These statements are well-known. We provide some short proofs here as we were unable to find any in the literature. For (i), let $R$ be a generator of $V_p$ satisfying $R^\sigma = \lambda(\sigma)R$ for all $\sigma \in G_K$. The point $R^\tau$ generates $V_p^\tau$. Let $\sigma \in G_K$. Then $\tau \sigma \tau^{-1} \in G_K$ and \[ R^{\tau \sigma \tau^{-1}} = \lambda(\tau \sigma \tau^{-1}) R.\] Applying $\tau$, we see that \[(R^\tau)^\sigma = \lambda(\tau \sigma \tau^{-1}) R^\tau. \] So the isogeny character of $(E^\tau,V_p^\tau)$ maps $\sigma$ to $\lambda(\tau \sigma \tau^{-1})$ as required.

For (ii), let $\varphi$ be the $K$-rational $p$-isogeny with kernel $V_p$. This means that $(E/V_p, E[p]/V_p) = (\varphi(E), \varphi(E[p])$. We fix a basis $(R_1,R_2)$ of $E[p]$ so that $R_1^\sigma = \lambda(\sigma)(R_1)$ for any $\sigma \in G_K$. Then $\varphi(R_2)$ generates $\varphi(E[p])$, and for any $\sigma \in G_K$ we have \[\varphi(R_2)^\sigma = \varphi^\sigma(R_2^\sigma) = \varphi\left(b_\sigma R_1 + (\chi_p \lambda^{-1})(\sigma)R_2 \right) = (\chi_p \lambda^{-1})(\sigma) \varphi(R_2),\] where $b_\sigma$ is the upper-right entry of the matrix $\overline{\rho}_{E,p}(\sigma)$ (with respect to the basis $(R_1,R_2)$).
\end{proof}

We will be particularly interested in studying the ramification of the character $\lambda^{12}$. For a prime $\mathfrak{p}$ of $K$ above $p$, we will denote by $I_\mathfrak{p}$ the inertia subgroup of $G_K$ corresponding to $\mathfrak{p}$.

\begin{proposition}[{\citep[Proposition~2.1]{irred}}]\label{sigdef}   Let $\lambda$ be the isogeny character of $(E,V_p)$.  Then $\lambda^{12}$ is unramified at the infinite primes of $K$ and at all finite primes of $K$ coprime to $p$. 
If $\mathfrak{p} \mid p$ is a prime of $K$, then there exists a unique integer $s_\mathfrak{p} \in \{0,12\}$ such that \[ \lambda^{12} |_{I_\mathfrak{p}} = (\chi_p |_{I_\mathfrak{p}})^{s_\mathfrak{p}}.\]
\end{proposition}

If $s_\mathfrak{p} = 0$ then we see that $\lambda^{12}$ is unramified at $\mathfrak{p}$. We now fix, once and for all, a prime $\mathfrak{p}_0 \mid p$ of $K$.  We define the \textbf{isogeny signature of $(E,V_p)$} to be $(s_{\mathfrak{p}_0}, s_{\tau(\mathfrak{\mathfrak{p}_0)}})$. We will also refer to this as the isogeny signature of the associated character $\lambda$. This isogeny signature is one of \[ (0,0),\, (12,12), \, (12,0), \,\, \text{or } (0,12). \] We will refer to $(0,0)$ and $(12,12)$ as \emph{constant} isogeny signatures, and we will refer to $(12,0)$ and $(0,12)$ as \emph{non-constant} isogeny signatures. We note that if the prime $p$ is inert in $K$, then the isogeny signature of $(E,V_p)$ is constant, since $\tau(\mathfrak{p}_0) = \mathfrak{p}_0$. Primes $p$ for which the isogeny signature of $(E,V_p)$ is constant are referred to as \emph{Type 1 primes} in \citep{Momose} and \citep{Banwait}. If the isogeny signature of $\lambda$ is $(0,0)$, then $\lambda^{12}$ is everywhere unramified.

\begin{remark} Our assumption that $E$ is semistable at the primes of $K$ above $p$ forces the integer $s_\mathfrak{p}$ appearing in Proposition \ref{sigdef} to be $0$ or $12$. Without assuming this semistability condition, $s_\mathfrak{p}$ can also take the values $4$, $6$, and $8$ (see \citep[pp.~7--9]{Criteres}). This gives rise to many more possible isogeny signatures. In particular, one of the isogeny signatures which must be considered is $(6,6)$. This is the isogeny signature which is the most difficult to rule out and it is the reason that the generalised Riemann hypothesis is assumed in \citep{Banwait} and \citep{B-D}. In the case that $K = \Q$, the case analogous to isogeny signature $(6,6)$ is considered by Mazur in \citep[pp.~154--155]{ratisog}, and is ruled out using some relatively elementary algebraic number theory to conclude that the class number of $\Q(\sqrt{-p})$ must be $1$.
\end{remark}

\begin{lemma}\label{sig2} Suppose the isogeny signature of $(E,V_p)$ is $(a,b)$. \begin{enumerate}[(i)]
\item The isogeny signature of $(E^\tau,V_p^\tau)$ is $(b,a)$.
\item The isogeny signature of $(E /V_p,E[p] / V_p)$ is $(12-a,12-b)$.
\item Let $\hat{E} / K$ be an elliptic curve with a $K$-rational subgroup $\hat{V}_p$ of order $p$. Suppose $\psi : E \rightarrow \hat{E}$ is an isomorphism (over $\overline{K}$) satisfying $\psi(V_p) = \hat{V}_p$. Then $(\hat{E}, \hat{V}_p)$ has isogeny signature $(a,b)$.
\end{enumerate}

\begin{proof} Parts (i) and (ii) follow from Lemma \ref{sig1} and the definition of isogeny signature. For (iii), the curve $\hat{E}$ will be a twist of the curve $E$ by a character, $\theta$, whose order divides $2$, $4$, or $6$. In particular, the order of $\theta$ divides $12$ and it follows that the twelfth powers of the isogeny characters of $(E,V_p)$ and $(\hat{E},\hat{V}_p)$ agree, and so the isogeny signatures must also agree. We refer to \citep[pp.~6--9]{twists} for more on how $\overline{\rho}_{E,p}$ is affected by twisting.  We note that (iii) is stated in \citep[p.~330]{Momose}.
\end{proof}
\end{lemma}

A pair $(E,V_p)$ gives rise to a $K$-rational point on $X_0(p)$. Part (iii) of Lemma \ref{sig2} allows us to extend the definition of isogeny signature to non-cuspidal points $y \in X_0(p)(K)$. We define the \textbf{isogeny signature of a non-cuspidal point $y \in X_0(p)(K)$} to be the isogeny signature of any pair $(\hat{E},\hat{V}_p)$ representing $y$, for $\hat{E}$ an elliptic curve defined over $K$ and $\hat{V}_p$ a $K$-rational subgroup of order $p$. 
If the isogeny signature of $y$ is $(a,b)$, then Parts (i) and (ii) of Lemma \ref{sig2} show that:
\begin{enumerate}[(i)]
\item the isogeny signature of $y^\tau$ is $(b,a)$; and
\item the isogeny signature of $w_p(y)$ is $(12-a,12-b)$.
\end{enumerate}
Here, $w_p$ denotes the Atkin--Lehner involution on $X_0(p)$. 


\section{Eliminating primes}


Throughout this section we will again assume that $p \geq 17$ and that $p$ is unramified in $K$. We continue to denote by $\mathfrak{p}_0$ a fixed prime of $K$ above $p$. We write $\mathcal{O}_K$ for the ring of integers of $K$. We write $\lambda$ for the isogeny character of $(E,V_p)$, and $(a,b)$ for the isogeny signature of $(E,V_p)$. 

For the remainder of the paper, we will write $\mathfrak{q}$ for a prime of $K$ above a rational prime $q$. We write $n_\mathfrak{q}$ for the norm of the ideal $\mathfrak{q}$, and we will denote by $\sigma_\mathfrak{q} \in G_K$ a Frobenius element at $\mathfrak{q}$.

Our aim in this section is to see how to reduce the number of possibilities for $p$ to a (small) finite set. Our strategy for bounding, and subsequently eliminating, $p$ is based on the following key result.
\begin{proposition}\label{Keyprop} Let $\lambda$ be the isogeny character of $(E,V_p)$ with isogeny signature $(a,b)$. Let $\mathfrak{q} \nmid p$ be a prime of $K$, let $r$ be the order of the class of $\mathfrak{q}$ in the class group of $K$, and write $\mathfrak{q}^r = \alpha \cdot \mathcal{O}_K $. Then \[\alpha^a \cdot (\alpha^\tau)^b \equiv \lambda^{12r}(\sigma_\mathfrak{q}) \pmod{\mathfrak{p}_0}. \] 
\end{proposition}

\begin{proof}  This is a direct consequence of \citep[Lemma~1]{Momose}. We refer also to \citep[Proposition~2.2]{irred} and \citep[Proposition~2.6]{Criteres} for statements that use notation closer to ours. Indeed, following \citep[Proposition~2.2]{irred}, the quantity $\mathcal{N}_\mathbf{s}(\alpha)$ is $\alpha^a \cdot (\alpha^\tau)^b$, and the prime $\mathfrak{q}$ is the unique prime in the support of $\alpha$.
\end{proof}

This proposition can then be used to prove the following result, which will be crucial in our proof of Theorem~\ref{Thm1}.

\begin{corollary}[{\citep[Corollary~3.2]{irred}}]\label{eps}
Let $K$ be a real quadratic field. Let $\epsilon$ be a fundamental unit of $K$. Suppose \[ p \nmid \mathrm{Norm}_{K / \Q}(\epsilon^{12}-1).\] Then the isogeny signature of $(E,V_p)$ is constant.
\end{corollary}

\begin{proof} Suppose the isogeny  signature of $(E,V_p)$ is $(12,0)$ or $(0,12)$ (i.e. non-constant). We will show that $p \mid \mathrm{Norm}_{K / \Q}(\epsilon^{12}-1)$. In the notation of \citep{irred}, $\mathcal{N}_{\mathbf{s}}(\epsilon) = \epsilon^{12}$ or $(\epsilon^\tau)^{12}$ (according to whether the isogeny signature is $(12,0)$ or $(0,12)$), and applying \citep[Corollary~3.2]{irred} gives that \[ \epsilon^{12} \equiv 1 \pmod{\mathfrak{p}_0} \quad \text{or} \quad (\epsilon^\tau)^{12} \equiv 1 \pmod{\mathfrak{p}_0}.\] Taking norms, we have that \[ p \mid \mathrm{Norm}_{K / \Q}(\epsilon^{12}-1)  \quad \text{or} \quad p \mid \mathrm{Norm}_{K / \Q}((\epsilon^\tau)^{12}-1). \]  The result follows since $\mathrm{Norm}_{K/ \Q}(\epsilon^{12}-1) = \mathrm{Norm}_{K/ \Q}((\epsilon^\tau)^{12}-1)$.
\end{proof}

This result will often allow us to focus on the case of a constant isogeny signature.  We will need different methods to deal with the non-constant isogeny signatures if $K$ is imaginary quadratic or if the prime factors of $\mathrm{Norm}_{K / \Q}(\epsilon^{12}-1)$ are large.

We can now also obtain a bound on $p$. The following result is similar to \citep[Theorem~2]{irred}. The key difference is that we have removed a factor of $2$ from the exponent of $3$ in the bound.

\begin{theorem}\label{Thm3} Let $K$ be a real quadratic field and let $\epsilon$ be a fundamental unit of $K$. Write $n$ for the exponent of the class group of $K$. Let $p$ be a prime such that there exists an elliptic curve $E / K$ which admits a $K$-rational $p$-isogeny and is semistable at all primes of $K$ above $p$. Then  either 
\begin{itemize}
\item $p$ ramifies in $K$; or
\item $p < (1+3^{6n})^2$; or
\item $p$ splits in $K$ and $p \mid\mathrm{Norm}_{K / \Q}(\epsilon^{12}-1)$.
\end{itemize}
Moreover, if $n=1$ then $p \equiv 1 \pmod{12}$ or $p \leq 19$.
\end{theorem}

\begin{proof} We assume $p \geq 17$ with $p$ unramified in $K$. We  let $V_p$ denote the kernel of the $K$-rational $p$-isogeny of $E$, and write $\lambda$ for the isogeny character of $(E,V_p)$. We assume that if $p$ splits in $K$ then $p \nmid\mathrm{Norm}_{K / \Q}(\epsilon^{12}-1)$ which ensures that the isogeny signature of $E$ is constant (by Corollary \ref{eps}). By interchanging $(E,V_p)$ with $(E/V_p, E[p]/V_p)$ if necessary, we may assume that the isogeny signature of $(E,V_p)$ is $(0,0)$. So $\lambda^{12}$ is everywhere unramified and it follows that $\lambda^{12n} = 1$. 

Let $M$ denote the field cut out by $\lambda^2$ (the fixed field of the kernel of $\lambda^2$), which will be an extension of $K$ of degree dividing $6n$, and therefore have absolute degree dividing $12n$. Then $\theta := \lambda |_{G_M}$ will either be trivial or a quadratic character. If $\theta = 1$ then $E$ has a point of order $P$ defined over $M$. Otherwise, we twist $E$ (viewed as a curve over $M$) by the quadratic character $\theta$, to obtain an elliptic curve with a point of order $p$ defined over $M$. We then apply Oesterl\'e's torsion bound \citep[Theorem~6.1]{DKSS}, to obtain  \[p < (1+3^{[M : \Q] /2 })^2  \leq (1+3^{6n})^2. \]

If $n = 1$, then we can obtain improved results. We have $\lambda^{12}=1$ and also $\lambda^{p-1} = 1$. So $\lambda^{\gcd(12,p-1)} = 1$. Therefore, if $p \not\equiv 1 \pmod{12}$ then $\lambda^4=1$ or $\lambda^6 = 1$. Applying the same argument as above, we conclude that there exists an elliptic curve with a point of order $p$ over a field of absolute degree dividing $4$ or $6$. Applying the torsion bounds of \citep[Theorem~1.2]{DKSS}, we conclude that $ p \leq 19$ or $p = 37$, and since $p \not\equiv 1 \pmod{12}$, we must have $p \leq 19$.
\end{proof}

\begin{remark} The idea used in this proof of applying a quadratic twist to reduce the degree of the field extension being considered is also used in \citep[p.~888]{realquad}.
\end{remark}

Although it is hidden within its proof, Corollary \ref{eps} (and consequently Theorem~\ref{Thm3}) relies on the fact that an elliptic curve will have a prime $\mathfrak{q}$ of potentially good reduction. In what follows, we will want to choose specific primes $\mathfrak{q}$, and we will not know whether they are of potentially good or of potentially multiplicative reduction for $E$. This leads us to separate our analysis into two cases.


\subsection{Primes of potentially good reduction}


Let $\mathfrak{q}$ be a prime of potentially good reduction for $E$. We will write $q$ for the rational prime below $\mathfrak{q}$. Let $r$ be the order of the class of $\mathfrak{q}$ in the class group of $K$, and write $\mathfrak{q}^r = \alpha \cdot \mathcal{O}_K$ (like in Proposition \ref{Keyprop}). We start by recalling some facts about the Frobenius element $\sigma_\mathfrak{q}$ and its action on the $p$-adic Tate module of $E$, following \citep[pp.~10-11]{Criteres}. The characteristic polynomial of Frobenius for $E$ at $\mathfrak{q}$ is given by \[P_\mathfrak{q}(X) = X^2-a_\mathfrak{q}(E)X + n_\mathfrak{q}. \] Let $\beta_1, \beta_2$ denote the roots of $P_\mathfrak{q}(X)$. Each root has absolute value $\sqrt{n_\mathfrak{q}}$. The two roots are complex conjugates, and we write $L = \Q(\beta_1)$ for the field they generate. The field $L$ is either $\Q$ or an imaginary quadratic field. Let $\mathcal{P}$ denote a prime of $L$ above $p$. Then $\beta_i \pmod{\mathcal{P}} \in \mathbb{F}_p^\times$ for $i \in \{1,2\}$, and moreover \[ \{ \lambda(\sigma_\mathfrak{q}), ( \chi_p \lambda^{-1}) (\sigma_\mathfrak{q})  \} = \{\beta_1 \Mod{\mathcal{P}}, \; \beta_2 \Mod{\mathcal{P}} \}. \]
The following result is a direct consequence of Proposition \ref{Keyprop}. We write $\mathrm{Res}$ for the resultant of two polynomials.

\begin{proposition}[{\citep[Lemma~3.1]{irred}}]\label{ResProp} Let $\lambda$ be the isogeny character of $(E,V_p)$ with isogeny signature $(a,b)$. Let $\mathfrak{q} \nmid p$ be a prime of potentially good reduction for $E$, let $r$ be the order of the class of $\mathfrak{q}$ in the class group of $K$, and write $ \mathfrak{q}^r = \alpha \cdot \mathcal{O}_K $. Then \[\mathfrak{p}_0 \mid \mathrm{Res}(P_\mathfrak{q}(X),X^{12r}-\alpha^a(\alpha^\tau)^b). \] If $(a,b) = (0,0)$, then \[ p \mid \mathrm{Res}(P_\mathfrak{q}(X),X^{12r}-1). \] 
\end{proposition}

The problem with applying this proposition is that the trace of Frobenius, $a_\mathfrak{q}(E)$, is unknown. However, we know that $ \abs{ a_\mathfrak{q}(E)} \leq 2 \sqrt{n_\mathfrak{q}}$. We define \begin{equation}\label{Aq} A_\mathfrak{q} := \{ a \in \Z :   \abs{a} \leq 2 \sqrt{n_\mathfrak{q}} \}. \end{equation}
Then $a_\mathfrak{q}(E) \in A_\mathfrak{q}$. The set $A_\mathfrak{q}$ only depends on $n_\mathfrak{q}$ and is therefore independent of the choice of prime $\mathfrak{q} \mid q$. We will therefore also write $A_q$ for this set.

\begin{remark} Instead of using the set $A_\mathfrak{q}$ defined in (\ref{Aq}), it is possible to run through all elliptic curves defined over the residue field of $\mathfrak{q}$ to compute a set of possible values for $a_\mathfrak{q}(E)$. This is possible because $E$ acquires good reduction at a totally ramified extension of the completion of $K$ at $\mathfrak{q}$. This idea is used, for example, in (parts of) \citep{Banwait}. However, this slows down the computations we will perform in Section 5, and we did not find it led to improved results in any of the cases we considered.
\end{remark}

Next, given an isogeny signature $(a,b) \ne (0,0)$, we  define \begin{equation*} R_\mathfrak{q} \coloneqq q \cdot \mathrm{lcm}_{a \in A_\mathfrak{q}}\left( \mathrm{Norm}_{K / \Q} \left( \mathrm{Res}(X^2 - aX + n_\mathfrak{q}, X^{12r}-\alpha^a (\alpha^\tau)^b )\right) \right).   \end{equation*}
If the isogeny signature $(a,b) = (0,0)$, then we simply define  \begin{equation}\label{Rq} R_\mathfrak{q} \coloneqq q \cdot \mathrm{lcm}_{a \in A_\mathfrak{q}}\left( \mathrm{Res}(X^2 - aX + n_\mathfrak{q}, X^{12r}-1)\right) .\end{equation}
In each case, $R_\mathfrak{q}$ is an integer. Also, $R_\mathfrak{q}$ is independent of the choice  of prime $\mathfrak{q}$ above $q$, and so we will also write $R_q$ for $R_\mathfrak{q}$.

\begin{corollary}\label{Rqcorol} Let $\lambda$ be the isogeny character of $(E,V_p)$ with isogeny signature $(a,b)$. Let $\mathfrak{q}$ be a prime of potentially good reduction for $E$. Then $p \mid R_\mathfrak{q}$. If $(a,b) = (0,0)$ then $R_\mathfrak{q} \neq 0$.
\end{corollary}

\begin{proof} If $\mathfrak{q} \nmid p$, then the main statement follows directly from Proposition \ref{ResProp}, and if $\mathfrak{q} \mid p$, then $p = q$ and so $p \mid R_\mathfrak{q}$ too, which is why we have included a factor of $q$ in the definition of $R_\mathfrak{q}$. Finally, if $(a,b) = (0,0)$ and $R_\mathfrak{q} = 0$, then for some $a \in A_\mathfrak{q}$, the roots of $X^2-aX+n_\mathfrak{q}$ (which are complex conjugate since $a \in A_\mathfrak{q}$) would both be roots of unity, and therefore their product, $n_\mathfrak{q}$, would also be a root of unity, a contradiction.
\end{proof}

The integer $R_\mathfrak{q}$ is independent of $p$. If $\mathfrak{q}_1, \dots, \mathfrak{q}_t$ are several primes of potentially good reduction for $E$, then \[ p \mid \gcd(R_{\mathfrak{q}_1}, \dots, R_{\mathfrak{q}_t}). \]
As we will see in Section 5, this idea will allow us to obtain a good bound on $p$. However, we have not yet used all the information at our disposal. We will now work with a fixed prime $p$ that we would like to eliminate and we suppose that $\mathfrak{q} \nmid p$. We first note that we can cut down the possibilities for $a_\mathfrak{q}(E)$. The roots of the characteristic polynomial of Frobenius reduce to elements of $\mathbb{F}_p^\times$, and so we have that \[ a_\mathfrak{q}(E)^2 - 4 \cdot n_\mathfrak{q} \; \; \text{ is a square mod } p. \] We define \begin{equation} A^{(p)}_\mathfrak{q} := \{ a \in \Z :  \abs{a} \leq 2 \sqrt{n_\mathfrak{q}} \text{ and } a^2-4n_\mathfrak{q} \Mod{p} \in \mathbb{F}_p^2  \}.  \end{equation} We have that $a_\mathfrak{q}(E) \in A^{(p)}_\mathfrak{q}$. Again,  $A^{(p)}_\mathfrak{q}$ is independent of the choice of prime $\mathfrak{q} \mid q$ and so we will also write $A^{(p)}_q$ for this set.

Now, by Proposition \ref{Keyprop}, we have that, \[(\lambda(\sigma_\mathfrak{q}))^{12r} = \alpha^a(\alpha^\tau)^b \pmod{\mathfrak{p}_0}, \] and this is then used to conclude that $\mathfrak{p}_0 \mid \mathrm{Res}(P_\mathfrak{q}(X),X^{12r}-\alpha^a(\alpha^\tau)^b)$, which gives Proposition \ref{ResProp}. However, recalling that $\chi_p(\sigma_\mathfrak{q}) \equiv n_\mathfrak{q} \pmod{p}$, we can also see that \[ \left((\chi_p\lambda^{-1})(\sigma_\mathfrak{q})\right)^{12r} = \chi_p(\sigma_\mathfrak{q})^{12r} \cdot \left(\lambda(\sigma_\mathfrak{q})^{12r}\right)^{-1} = \frac{n_\mathfrak{q}^{12r}}{\alpha^a (\alpha^\tau)^b} \pmod{\mathfrak{p}_0}. \] Then $n_\mathfrak{q}^{12r} = \mathrm{Norm}_{K / \Q}(\alpha)^{12} = (\alpha \alpha^\tau)^{12}$. So \[ \left((\chi_p\lambda^{-1})(\sigma_\mathfrak{q})\right)^{12r} = \alpha^{12-a} (\alpha^\tau)^{12-b} \pmod{\mathfrak{p}_0}. \] An alternative way of seeing this is by swapping $(E,V_p)$ with $(E/V_p,E[p]/V_p)$ and using Proposition \ref{Keyprop} along with Lemma \ref{sig1}.

For $a \in A^{(p)}_\mathfrak{q}$, let $\{\gamma_{a,1}, \gamma_{a,2} \} \subset \mathbb{F}_p^\times$ denote the reductions of the roots of  $X^2 - aX+n_\mathfrak{q}$. Since $a_\mathfrak{q}(E) \in A^{(p)}_\mathfrak{q}$, we must have that  \[ \{\gamma_{a,1}, \gamma_{a,2} \}  = \{ \lambda(\sigma_\mathfrak{q}), (\chi_p \lambda^{-1})(\sigma_\mathfrak{q}) \} \text{ for some } a \in A^{(p)}_\mathfrak{q}. \] 

\begin{lemma}\label{pcrit} Let $(E,V_p)$ have isogeny signature $(a,b)$. Let $\mathfrak{q} \nmid p$ be a prime of $K$ of potentially good reduction for $E$, let $r$ be the order of the class of $\mathfrak{q}$ in the class group of $K$, and write $ \mathfrak{q}^r = \alpha \cdot \mathcal{O}_K$. Then for some $a \in A^{(p)}_\mathfrak{q}$, (at least) one of the following two conditions holds: \begin{enumerate}[(i)]
\item $\gamma_{a,1}^{12r} = \alpha^a(\alpha^\tau)^b \pmod{\mathfrak{p}_0}$ and $\gamma_{a,2}^{12r} =  \alpha^{12-a} (\alpha^\tau)^{12-b} \pmod{\mathfrak{p}_0}$; or
\item  $\gamma_{a,2}^{12r} = \alpha^a(\alpha^\tau)^b \pmod{\mathfrak{p}_0}$ and $\gamma_{a,1}^{12r} =  \alpha^{12-a} (\alpha^\tau)^{12-b} \pmod{\mathfrak{p}_0}$.
\end{enumerate}
If $(a,b)= (0,0)$ then these conditions simplify, and we have that for some $a \in A^{(p)}_\mathfrak{q}$:  \begin{enumerate}[(i)]
\item $\gamma_{a,1}^{12r} = 1 \pmod{p}$ and $\gamma_{a,2}^{12r} = n_\mathfrak{q}^{12r} \pmod{p}$; or
\item $\gamma_{a,2}^{12r} = 1 \pmod{p}$ and $\gamma_{a,1}^{12r} = n_\mathfrak{q}^{12r} \pmod{p}$.
\end{enumerate}
\end{lemma}

This gives us a strategy for eliminating a given prime $p$. Indeed, if for all $a \in A^{(p)}_\mathfrak{q}$, (at least) one of the conditions in (i) is not satisfied \textbf{and} (at least) one of the conditions in (ii) is not satisfied, then we obtain a contradiction.

\begin{remark} In \citep[p.~338]{Momose}, conditions analogous to (i) and (ii) of Lemma \ref{pcrit} in the case of isogeny signature $(0,0)$ are effectively combined to say that \[ \gamma_{a,1}^{12h} + \gamma_{a,2}^{12h} = 1 + n_\mathfrak{q}^{12r} \pmod{p},\] where $h$ is the class number of $K$. This restores a certain symmetry and is sufficient to bound $p$. However, combining the two conditions places fewer conditions on $p$ and reduces the chances of eliminating the prime $p$. 
\end{remark}


\subsection{Primes of potentially multiplicative reduction}


Let $\mathfrak{q}$ be a prime of potentially multiplicative reduction for $E$. As before, we will write $q$ for the rational prime below $\mathfrak{q}$, we let $r$ be the order of the class of $\mathfrak{q}$ in the class group of $K$, and write $\mathfrak{q}^r = \alpha \cdot \mathcal{O}_K$. We would like to obtain results analogous to Proposition \ref{ResProp} and Corollary \ref{Rqcorol} for primes of potentially multiplicative reduction. If the isogeny signature of $(E,V_p)$ is $(a,b)$, then we start by defining \begin{equation}\label{Mq} M_\mathfrak{q} \coloneqq q \cdot \mathrm{Norm}_{K / \Q} \left( (\alpha^a (\alpha^\tau)^b -1) \cdot (\alpha^a (\alpha^\tau)^b-n_\mathfrak{q}^{12r}) \right). \end{equation} The integer $M_\mathfrak{q}$ is independent of the choice of prime $\mathfrak{q} \mid q$, and so we will also write is as $M_q$.

\begin{proposition} \label{MultProp} Let $(E,V_p)$ have isogeny signature $(a,b)$. Let $\mathfrak{q}$ be a prime of $K$ of potentially multiplicative reduction for $E$. Then $p \mid M_\mathfrak{q}$. If $(a,b) = (12,0)$ or $(0,12)$ then $M_\mathfrak{q} \ne 0$.
\end{proposition}

\begin{proof}  Let $\lambda$ be the isogeny character of $(E,V_p)$, with isogeny signature $(a,b)$. If $\mathfrak{q} \mid p$ then the statement holds, so we will assume that $\mathfrak{q} \nmid p$. Let $r$ be the order of the class of $\mathfrak{q}$ in the class group of $K$, and write $\mathfrak{q}^r = \alpha \cdot \mathcal{O}_K$. We then have that (see \citep[Proposition~1.4]{Criteres} for example)
 \[ \lambda^2(\sigma_\mathfrak{q}) \equiv 1 \text{ or } n_\mathfrak{q}^2 \pmod{p}. \] Then, applying Proposition \ref{Keyprop}, we see that \[ \alpha^a (\alpha^\tau)^b \equiv \lambda^{12r}(\sigma_\mathfrak{q}) \equiv 1 \text{ or } n_\mathfrak{q}^{12r} \pmod{\mathfrak{p}_0}. \] Taking norms, we see that \[ p \mid \mathrm{Norm}_{K / \Q}(\alpha^a (\alpha^\tau)^b  - 1) \text{ or } \mathrm{Norm}_{K / \Q}(\alpha^a (\alpha^\tau)^b  - n_\mathfrak{q}^{12r}), \] and we conclude that $p \mid M_\mathfrak{q}$. Finally, it is clear that $M_\mathfrak{q} \neq 0$ if the isogeny signature of $(E,V_p)$ is non-constant. \end{proof}

Unfortunately, $M_\mathfrak{q} = 0$ for all primes $\mathfrak{q}$ if the isogeny signature of $(E,V_p)$ is constant, and so this result will not help us eliminate primes $p$ in the case of a constant isogeny signature. We will use a different approach for this case.

As in Section 3 we will write $x \in X_0(p)(K)$ for the non-cuspidal point that the pair $(E,V_p)$ gives rise to, and we recall that we extended the notion of isogeny signature to non-cuspidal points $y \in X_0(p)(K)$. We will denote the two cusps of $X_0(p)$ (which are defined over $\Q$) by $\infty$ and $0$. We write $k_\mathfrak{q}$ for the residue field of $\mathfrak{q}$. If $y \in X_0(p)(K)$, we denote by $y_{\scriptscriptstyle k_\mathfrak{q}}$ the reduction of $y$ mod $\mathfrak{q}$. Since $\mathfrak{q}$ is a prime of potentially multiplicative reduction for $E$, we have that \begin{equation}\label{redcusp1} x_{\scriptscriptstyle k_\mathfrak{q}} = \infty_{\scriptscriptstyle k_\mathfrak{q}} \text{ or } 0_{\scriptscriptstyle k_\mathfrak{q}}. \end{equation}

We will now assume that $E$ has potentially multiplicative reduction at all primes $\mathfrak{q} \mid q$. Instead of working with the point $x \in X_0(p)(K)$, we will instead focus on the pair $(x,x^\tau)$, which we view as a rational point on the symmetric square of $X_0(p)$, which we write as \[ X_0(p)^{(2)} = \frac{ X_0(p) \times X_0(p) }{ \mathrm{Sym}_2}. \] We will denote by $(x,x^\tau)_{\scriptscriptstyle \mathbb{F}_q}$ the reduction of this rational point mod $q$. From (\ref{redcusp1}), and our assumption that all primes above $q$ are of potentially multiplicative reduction for $E$, we have that \begin{equation} (x,x^\tau)_{\scriptscriptstyle \mathbb{F}_q} = (\infty,\infty)_{\scriptscriptstyle \mathbb{F}_q}, \; (0,0)_{\scriptscriptstyle \mathbb{F}_q}, \text{ or } (\infty,0)_{\scriptscriptstyle \mathbb{F}_q}. \end{equation} If the prime $q$ does not split in $K$ then there is a unique prime above  $q$ and it follows that $(x,x^\tau)_{\scriptscriptstyle \mathbb{F}_q} = ( \infty,\infty)_{\scriptscriptstyle \mathbb{F}_q} \text{ or } (0,0)_{\scriptscriptstyle \mathbb{F}_q}.$ We start by stating the following result (for which we do not need to assume that $p$ is unramified in $K$).

\begin{proposition}[{\citep[p.~32]{B-D}}]\label{immersion} Let $(q,p)$ be a pair of primes satisfying one of the following pairs of conditions:
\begin{itemize}
\item $q \neq 2, p$ and $p \geq 23$, $p \neq 37$; or
\item $q = 2$ and $ 23 \leq p \leq 2357$, $p \neq 37, 41$. 
\end{itemize} Let $y \in X_0(p)(K)$ and suppose  $(y,y^\tau)_{\mathbb{F}_q} = (\infty,\infty)_{\scriptscriptstyle \mathbb{F}_q}$ or $(0,0)_{\scriptscriptstyle \mathbb{F}_q}.$ Then $y= \infty $ or $y = 0$.
\end{proposition}

\begin{proof} By applying the Atkin--Lehner involution $w_p$ (which swaps the cusps) to $(y,y^\tau)$ if necessary, we may assume that  $(y,y^\tau)_{\mathbb{F}_q} = (\infty,\infty)_{\scriptscriptstyle \mathbb{F}_q}$, and we aim to prove that $(y,y^\tau) = (\infty,\infty)$. We introduce the following notation. We write $S = \mathrm{Spec}(\Z[1/p])$, we write $J_0(p)$ for the Jacobian of $X_0(p)$, and we write $J_e(p)$ for the \emph{winding quotient} of $J_0(p)$ (defined in as in \citep[Definition~2.1]{DKSS}). We then consider the map 
\[f_p:  X^{(2)}_0(p)_{/S} 
\longrightarrow  J_0(p)_{/S} \longrightarrow J_e(p)_{/ S}, \] which is the composition of the Abel--Jacobi map with basepoint $2\infty$ and the projection map to the winding quotient. This is the same map as the one considered in \citep[p.~223]{kamienny}, except that we project to the winding quotient instead of the Eisenstein quotient. 

In order to prove that $(y,y^\tau) = (\infty,\infty)$, following the argument of \citep[p.~225]{kamienny}, it suffices to verify that the map $f_p$ is a \emph{formal immersion} along $(\infty, \infty)$ in characteristic $q$.  This is precisely what is done in \citep[p.~29--33]{B-D} (in particular we refer the reader to the $d = 2$ row in Table 7 of this paper when $q >2$ and the associated data file for the case when $q=2$). As noted in \citep[p.~32]{B-D}, this computation is really an extension of \citep[Lemma~5.4]{DKSS} to the case of quadratic fields.
\end{proof}

We will describe any pair of primes $(q,p)$ satisfying the conditions of Proposition \ref{immersion} as an \textbf{admissible pair}. The upper bound of $2357$ on $p$ in the case $q = 2$ is simply taken from \citep[p.~32]{B-D}. We expect this to hold for all $p > 2357$ and this bound could most likely be increased if desired. Being able to work with the prime $q = 2$ will provide useful information for the computations we carry out in the next section. 

Proposition \ref{immersion} already tells us that if $(q,p)$ is an admissible pair with $q$ a prime that does not split in $K$, and the unique prime of $K$ above $q$ is of potentially multiplicative reduction for $E$, then $E$ cannot have a $K$-rational $p$-isogeny for $p \geq 23$ with $ p \neq 37$. However, it does not consider the case  $(y,y^\tau)_{\scriptscriptstyle  \mathbb{F}_q} = (\infty, 0)_{\scriptscriptstyle  \mathbb{F}_q}$, which is certainly possible if $q$ splits in $K$. Our next result, from the author's own work, addresses this case (for which we do not need to assume that $p$ is unramified in $K$).

\begin{lemma}[{\citep[Lemma~4.8]{MJ}}]\label{redcusp3} Let $p$ and $q$ be primes. Let $y \in X_0(p)(K)$. Suppose $(y,y^\tau)_{\scriptscriptstyle  \mathbb{F}_q} = (\infty, 0)_{\scriptscriptstyle  \mathbb{F}_q}$. Suppose $q \ne 2,p$ and $p \geq 23$. Then $w_p(y) = y^\tau$.
\end{lemma}

If the isogeny signature of $x \in X_0(p)(K)$ is constant, then the isogeny signatures of the points $w_p(x)$ and $x^\tau$ differ (by Lemma \ref{sig2}), so $w_p(x) \neq x^\tau$. This observation combined with Proposition \ref{immersion} and Lemma \ref{redcusp3} gives the following result.

\begin{proposition}\label{red} Let $(E,V_p)$ be an elliptic curve over $K$ with a $K$-rational subgroup of order $p$. Let $q$ be a prime for which all primes of $K$ above $q$ are of potentially multiplicative reduction for $E$. Suppose that $(q,p)$ is an admissible pair, and that if $q = 2$ then $q$ is inert or ramified in $K$. Then the isogeny signature of $(E,V_p)$ is non-constant.
\end{proposition}

In order to deal with the case in which $q = 2$ and $2$ splits in $K$, we will use the following result.

\begin{lemma} \label{red2} Suppose that the isogeny signature of $(E,V_p)$ is constant. Suppose that $(2,p)$ is an admissible pair such that $2$ splits in $K$ and both primes of $K$ above $2$ are of potentially multiplicative reduction for $E$. Let $\mathfrak{q}_2$ denote a prime of $K$ above $2$ and let $r$ be the order of the class of $\mathfrak{q}_2$ in the class group of $K$. Then  \[ p \mid 2^{12r}-1. \]
\end{lemma}

\begin{proof} This may be viewed as a special case of part of \citep[Proposition~4.1]{Banwait}. We may assume, by replacing $(E,V_p)$ with $(E/V_p,E[p]/V_p)$ if necessary, that the isogeny signature of $(E,V_p)$ is $(0,0)$. By Proposition \ref{immersion}, the point $(x,x^\tau) \in X_0(p)^{(2)}(\Q)$ that $(E,V_p)$ gives rise to must satisfy $(x,x^\tau)_{\scriptscriptstyle  \mathbb{F}_2} = (\infty, 0)_{\scriptscriptstyle  \mathbb{F}_2}$. It follows that either $x_{\scriptscriptstyle k_{\mathfrak{q}_2}} = 0_{\scriptscriptstyle k_{\mathfrak{q}_2}}$ or $x^\tau_{\scriptscriptstyle k_{\mathfrak{q}_2}} = 0_{\scriptscriptstyle k_{\mathfrak{q}_2}}$. By replacing $(E,V_p)$ by $(E^\tau,V_p^\tau)$ if necessary, we may assume that $x_{\scriptscriptstyle k_{\mathfrak{q}_2}} = 0_{\scriptscriptstyle k_{\mathfrak{q}_2}}$. We note that the isogeny signature remains $(0,0)$.

From the proof of Proposition \ref{MultProp} we know that \[ \lambda^2(\sigma_{\mathfrak{q_2}}) \equiv 1 \text{ or } n_{\mathfrak{q}_2}^2 \pmod{p}. \]
Since $x_{\scriptscriptstyle k_{\mathfrak{q}_2}} = 0_{\scriptscriptstyle k_{\mathfrak{q}_2}}$, applying the argument of \citep[p.~19]{Banwait} we must be in the second case: $\lambda^2(\sigma_{\mathfrak{q_2}}) \equiv  n_{\mathfrak{q}_2}^2 \pmod{p}$. Since the isogeny signature of $(E,V_p)$ is $(0,0)$, using Proposition \ref{Keyprop} we obtain \[ 1 \equiv \lambda^{12r}(\sigma_{\mathfrak{q_2}}) \equiv n_{\mathfrak{q}_2}^{12r} \equiv 2^{12r} \pmod{p}, \] where we have used the fact that $n_{\mathfrak{q}_2} = 2$ in the final step.
\end{proof}

\begin{remark} In \citep[p.~338]{Momose}, it is claimed that if the isogeny signature of $(E,V_p)$ is constant, if $\mathfrak{q} \mid q$ is of potentially multiplicative reduction for $E$, and if $(q,p)$ is an admissible pair, then $p-1 \mid 12h$, where $h$ is the class number of $K$. The argument leading to this seems to be incorrect. A correct condition is $p \mid n_\mathfrak{q}^{12r}-1$, where $r$ is the order of the class of $\mathfrak{q}$ in the class group of $K$, which is part of \citep[Proposition~4.1]{Banwait} (referred to in the proof of Lemma \ref{red2} above).
\end{remark}

\begin{remark} It is perhaps worth noting at this point that the results of this section could be suitably extended to number fields of larger degree (we refer to \citep{B-D} for a selection of such results). There are two main reasons we have chosen to focus on the case of quadratic fields. Firstly, we can produce strong results in the case of quadratic fields. This is due to the fact that the methods we use turn out to be very effective in the case of quadratic fields. It is also due to the extensive work done on studying quadratic points on the modular curves $X_0(p)$ of small genus. Secondly, as discussed in the introduction, applications of the modular approach (for solving Diophantine equations) over totally real fields are most common over (real) quadratic fields, and so we hope that our results will be directly useful in this setting.
\end{remark}

\section{Computations}


In this section, we apply the results of Section 4 to certain specific quadratic fields and families of quadratic fields. We start by outlining the basic strategy. 

As in Sections 3 and 4, we let $p$ be a prime and let $(E,V_p)$ be an elliptic curve over a quadratic field $K$ with a $K$-rational subgroup of order $p$, and we assume that $E$ is semistable at the primes of $K$ above $p$. For the moment, we do not make any further assumptions on $p$ (in particular, $p$ could ramify in $K$, or $p$ could be $<17$). Our strategy consists of three main steps. In each step we try and eliminate the prime $p$ as a possible $K$-rational $p$-isogeny prime for $E$.

\textbf{Step 1.} Assume that $p \geq 23$, $p \neq 37$, that $p$ is unramified in $K$, and that the isogeny signature of $(E,V_p)$ is constant. 

Requiring $p \geq 23$ and $p \neq 37$ means that $(q,p)$ will be an admissible pair for any $q \neq 2,p,$ and for $q=2$ when $p \leq 2357$ and $p \neq 41$. We may assume, by replacing $(E,V_p)$ by $(E/V_p,E[p]/V_p)$ if necessary, that the isogeny signature of $(E,V_p)$ is $(0,0)$.

We now choose auxiliary primes $q_1, \dots, q_t$, with $q_i \geq 3$ for all $i$. By Proposition \ref{red}, unless $q_i=p$, it is not possible for both primes of $K$ above $q_i$ to be of potentially multiplicative reduction for $E$, and so there is a prime $\mathfrak{q}_i \mid q_i$ of potentially good reduction for $E$. We compute the integers $R_{q_i}$ (which we recall are independent of the prime chosen above $q_i$), and applying Corollary \ref{Rqcorol} we have that 
 \[ p \mid \gcd(R_{\mathfrak{q}_1}, \dots, R_{\mathfrak{q}_t}). \] If $q_i = p$ for some $i$, then $p \mid R_{q_i}$ and so this still holds. This leaves us with a finite set of primes $p$ which we are unable to eliminate (which we hope is fairly small). 

For each remaining prime, we may then perform a finer analysis to try and eliminate it. We use Lemma \ref{pcrit} to try and achieve a contradiction with each prime $\mathfrak{q}_i$. Further, if $(2,p)$ is an admissible pair then we may also use $q=2$ to try and eliminate $p$. Indeed, if $2$ is inert or ramified in $K$, then by Proposition \ref{red}, the unique prime of $K$ above $2$ must be of potentially good reduction for $E$, and we may apply Lemma \ref{pcrit}. Otherwise, we may apply Lemma \ref{pcrit} in combination with Lemma \ref{red2}.

\textbf{Step 2.} Assume that $p \geq 17$, that $p$ is unramified in $K$, and that the isogeny signature of $(E,V_p)$ is non-constant. 

In this case, $p$ must split in $K$. Also, by Cororally \ref{eps}, if $K$ is a real quadratic field with fundamental unit $\epsilon$, then $ p \mid \mathrm{Norm}_{K / \Q}(\epsilon^{12}-1)$. 

We now use auxiliary primes $q_1, \dots, q_t \geq 2$. For each auxiliary prime $q$, either there is a prime of $K$ above $q$ which is of potentially good reduction for $E$, in which case $p \mid R_q$ by Corollary \ref{Rqcorol}, or there is a prime of $K$ above $q$ which is of potentially multiplicative reduction for $E$, in which case $p \mid M_q$ by Proposition \ref{MultProp}. In both cases, we have $p \mid R_qM_q$, and $R_qM_q$ is independent of $p$ for each $q$.

\textbf{Step 3.} We consider the primes $p = 11, 17,$ and $19$, the primes that ramify in $K$, and any primes that we were unable to eliminate in Step 1 and Step 2.

For each of these primes we study $X_0(p)(K)$ directly to try and eliminate it. 

We note that it is not possible to eliminate the primes $2,3,5,7,13,$ and $37$. To see this, it is enough to consider the base change to $K$ of the elliptic curves appearing in Table \ref{TabRat}.

The following lemma will be useful in Step 1 and Step 3.

\begin{lemma}\label{exceptional} Let $K= \Q(\sqrt{d})$ be a quadratic field and suppose \[ d \notin \{ -1, -3, -5,-7,-11,-15,-31,-71,-131 \}. \] Let $x \in X_0(p)(K)$ be a non-cuspidal point and suppose $w_p(x) \neq x^\tau$ (which will be the case if the isogeny signature of $x$ is constant). Then \begin{equation}\label{lst} p \notin \{ 23,29,31,41,43,47,53,59,61,67,73, 79, 83, 89, 101, 131\}. \end{equation}
\end{lemma}

\begin{proof} Suppose $p$ is one of the primes in (\ref{lst}). For $p \leq 73$ the papers \citep{hyperquad} and \citep{Boxquad} compute all quadratic points $x \in X_0(p)(K)$ that satisfy $w_p(x) \neq x^\tau$; such points are called \emph{exceptional}. The recent paper \citep{NajVuk} does the same for $p \geq 79$, and  we will use this result for $p = 101$ in our example in Section 5.3. In each paper, we simply consult the tables and read off the possible fields over which these exceptional quadratic points are defined. 
\end{proof}

\subsection{Families of quadratic fields}


We start by proving Theorem \ref{Thm1}.

\begin{proof}[Proof of Theorem \ref{Thm1}] We assume that there there exists an elliptic curve $E / K$ which admits a $K$-rational $p$-isogeny and is semistable at the primes of $K$ above $p$. We apply the strategy described at the start of this section. Even though the quadratic field $K$ is not fixed, knowing the exponent, $n$, of the class group of $K$ will be enough to do this. We will assume that $p \geq 23$, $p \neq 37$, that $p$ is unramified in $K$, and that if $p$ splits in $K$ then $p \nmid \mathrm{Norm}_{K / \Q}(\epsilon^{12}-1)$. These assumptions on the prime $p$ mean that the isogeny signature of $(E,V_p)$ is constant, and we need only focus on Step 1 of the strategy outlined above. We are aiming to achieve a contradiction.

We will choose several auxiliary primes $q$. For each auxiliary prime $q$ we choose, and a prime $\mathfrak{q} \mid q$, we \emph{do not} know two things: \begin{itemize}
\item whether $q$ splits, is inert, or ramifies in $K$; and
\item the order, $r$, of the class of $\mathfrak{q}$ in the class group of $K$.
\end{itemize}
If $q$ is inert in $K$, then $r = 1$ and $n_\mathfrak{q} = q^2$. Otherwise, $q$ is split or ramified in $K$, $n_\mathfrak{q} = q$,  and $r \mid n$. This means that for each $\mathfrak{q}$, we have $1+d(n)$ possibilities for the pair $(n_\mathfrak{q}, r)$, where $d(n)$ denotes the number of positive divisors of $n$. Computing the integer $R_\mathfrak{q}$ defined in (\ref{Rq}) only requires this pair as input, and so we simply run through all $1+d(n)$ possibilities, obtain an integer $R_\mathfrak{q}$ for each of these, and take their lowest common multiple at the end. We may then apply Lemma \ref{pcrit} to try and eliminate even more primes, again running through all $1+d(n)$ possibilities. We use the auxiliary prime $q = 2$ if $(2,p)$ is an admissible pair.  
Using the auxiliary primes $3 \leq q \leq 19$, followed by $q=2$, we were able to eliminate all possibilities for the prime $p$, other than $p = 73$ in the case $n = 3$, which we rule out by applying Lemma \ref{exceptional}.
\end{proof}

We note that it took under a tenth of a second to perform the necessary computations to obtain this result. 

We have considered $n \leq 3$ here, but we can also prove similar results for larger values of $n$, with the caveat that the size of the set of constant isogeny signature primes which we are unable to eliminate sometimes increases. For example, the constant isogeny signature primes we are unable to eliminate in the case $n = 100$ are \[ p \in \{97, 151, 241, 401, 601, 1201, 1801  \}.\] We  expect that we should be able to eliminate these extra primes, but the method we use does not achieve this. 

\begin{remark}\label{conv_rem}
In contrast to the case $n=100$ above, as noted in the introduction, it is not possible to eliminate the primes $p \leq 19$ or $p = 37$ appearing in Theorem \ref{Thm1}, and we may construct the appropriate elliptic curves to verify this. Indeed, for a given field $K$, if $p \in \{2,3,5,7,13,37\}$ then we may use the base change to $K$ of the corresponding elliptic curve appearing in Table \ref{TabRat}. For $p \in \{11,17,19\}$ we may search for a quadratic field $K$ for which the elliptic curve $X_0(p)$ has positive rank over $K$ and follow the strategy of the proof of Theorem \ref{Thm2} below. For $n = 1$, $2$, and $3$, it turns out we can use the fields $K = \Q(\sqrt{29}), \Q(\sqrt{10}),$ and $\Q(\sqrt{79})$ respectively. Alternatively, for $n=1$ we could use Theorem \ref{Thm2} and use the fields $K = \Q(\sqrt{2})$ and $\Q(\sqrt{3})$. These computations are presented in the accompanying \texttt{Magma} files. 
\end{remark}

The following result demonstrates that knowing the splitting behaviour of certain primes can produce strong results.

\begin{theorem}\label{Thm4}
Let $K$ be a real quadratic field in which the primes $2$ and $3$ are inert, and let $\epsilon$ be a fundamental unit of $K$. Let $p$ be a prime such that there exists an elliptic curve $E / K$ which admits a $K$-rational $p$-isogeny and is semistable at all primes of $K$ above $p$. Then either 
\begin{itemize}
\item $p$ ramifies in $K$; or
\item $p \in \{2,3,5,7,11,13,17,19,37 \}$; or
\item $p$ splits in $K$ and $p \mid \mathrm{Norm}_{K / \Q}(\epsilon^{12}-1)$.
\end{itemize}
\end{theorem}

\begin{proof} We proceed as in the proof of Theorem \ref{Thm1} and need only consider Step 1 of the strategy outlined at the start of this section. We assume that $p \geq 23$, $p \neq 37$, $p$ is unramified in $K$, and that the isogeny signature of $(E,V_p)$ is constant. We aim to obtain a contradiction. We start by using the auxiliary prime $q=3$. Since $3$ is inert in $K$, we know that the unique prime above $3$ is a principal ideal and has norm $3^2$. By Proposition \ref{red}, the unique prime above $3$ must be of potentially good reduction for $E$ and so $p \mid R_3$. The largest prime factor of $R_3$ is $1489$ and $41 \nmid R_3$, so we may also now use the auxiliary prime $q=2$, since $(2,p)$ will be admissible for all remaining primes $p$. We again apply Proposition \ref{red} to conclude that the unique prime above $2$ must be of potentially good reduction for $E$, and so $p \mid R_2$. So \[ p \mid \gcd(R_3,R_2) = 754471972800 = 2^6 \cdot 3^4 \cdot 5^2 \cdot 7^2 \cdot 13^2 \cdot 19 \cdot 37, \] giving the required contradiction.
\end{proof}

Similarly to Theorem \ref{Thm1} above, we cannot eliminate the primes $p \leq 19$ or $p = 37$ from the statement of this theorem. This can be seen by considering the field $K= \Q(\sqrt{29})$ again (as in Remark \ref{conv_rem}) in which the primes $2$ and $3$ are inert.


\subsection{Small real quadratic fields}


We will consider the real quadratic fields $K = \Q(\sqrt{d})$ for $d \in \{2,3,5,6,7\}$. Each of these fields has class number $1$. 

\begin{lemma}\label{lem2} Let $K = \Q(\sqrt{d})$, where $ d \in \{2,3,5,6,7\}$. Let $p$ be a prime such that there exists an elliptic curve $E / K$ which admits a $K$-rational $p$-isogeny and is semistable at all primes of $K$ above $p$. Then either  $p \in \{2,3,5,7,13,37\}$ or the pair $(d,p)$ appears in Table \ref{Tab2} of the introduction.
\end{lemma}

In order to prove Theorem \ref{Thm2} (for $d > 0$), we will also need to prove the converse of this statement. We do this afterwards.

\begin{proof}

We start by applying Theorem \ref{Thm1}. It remains to consider the primes $p \in \{11, 17, 19 \}$ and the primes that split in $K$ that divide $\mathrm{Norm}_{K / \Q}(\epsilon^{12} - 1)$, for $\epsilon$ a fundamental unit of $K$. 

The only field $K$ in our list for which there exists a prime $p \geq 23$, $p \neq 37$ that splits in $K$ and divides $\mathrm{Norm}(\epsilon^{12} - 1)$ is $K = \Q(\sqrt{6})$. In this case, the element $\epsilon = 5 + 2 \sqrt{6}$ is a fundamental unit for $K$. The prime factors of $\mathrm{Norm}(\epsilon^{12}-1)$ are $2,3,5, 11,$ and $97$, and we must therefore consider $p=97$, which splits in $K$. We now follow Step 2 of the strategy outlined at the start of the section. By replacing $(E,V_p)$ with $(E/V_p,E[p]/V_p)$ (or by $(E^\tau,V_p^\tau)$) if necessary, we may assume that the isogeny signature of $(E,V_p)$ is $(12,0)$. We  will apply  Corollary \ref{Rqcorol} and Proposition \ref{MultProp} using the auxiliary prime $q=5$ to obtain a contradiction. We compute the integers $R_5$ and $M_5$, which we recall are independent of the prime chosen above $5$. We compute that \begin{align*} R_5  & = 2^{10} \cdot 3^{10} \cdot 5^{15} \cdot 13^2 \cdot 17^2 \cdot 19^4 \cdot 23^2 \cdot 41^2 \cdot 73^2 \cdot 241^2, \\ M_5 & = 2^{10} \cdot 3^8 \cdot 5^{15} \cdot 43^2 \cdot 433^2. \end{align*} Since $97$ is not a prime factor of $R_5$ or $M_5$ we may eliminate the prime $p=97$ for the case of a non-constant isogeny signature.

We now continue with Step 3 and consider the primes $11, 17,$ and $19$. We will consider the case $K = \Q(\sqrt{6})$, the other cases being similar. We must eliminate the prime $p = 19$. We start by computing that $X_0(19)(K) = X_0(19)(\Q) = \Z / 3\Z$. Let $x \in X_0(19)(K)$ denote the point corresponding to $(E,V_p)$, which is the unique non-cupsidal point in $X_0(19)(K)$. We now use exactly the same argument as in the proof Corollary \ref{ratcorol}. Let $\mathfrak{p}_1$ and $\mathfrak{p}_2$ denote the two primes of $K$ above $19$. The curve $E$ is the quadratic twist of a curve $F$, defined over $K$, which has potentially good reduction at $\mathfrak{p}_1$ and $\mathfrak{p}_2$ and satisfies $0 < v_{\mathfrak{p}_i}(\Delta_\mathrm{min}(F)) < 6$ for $i \in \{1,2 \}$. It follows that $v_{\mathfrak{p}_i}(\Delta_\mathrm{min}(E))>0,$ so $E$ must have potentially good, non-semistable, reduction at $\mathfrak{p}_1$ and $\mathfrak{p}_2$, which allows us to eliminate $p = 19$.
\end{proof}

\begin{proof}[Proof of Theorem \ref{Thm2} for $d>0$]
By Lemma \ref{lem2} it will be enough to find an appropriate elliptic curve for each value of $p$. For $p \in \{2,3,5,7,13,37 \}$ we may simply use the base change to $K$ of the elliptic curve appearing in Table \ref{TabRat}. It remains to deal with the primes $p \in \{11, 17,19 \}$ in Table \ref{Tab2}. In each case, $X_0(p)$ is an elliptic curve, and using \texttt{Magma} we can directly compute that $X_0(p)(K)$ has rank $1$, along with a generator, $Q$, for the free part of its Mordell--Weil group. We may then write down, using \texttt{Magma}'s `small modular curve package', an elliptic curve $E / K$ with a $K$-rational $p$-isogeny representing the point $mQ$ for (small) integers $m$, and test its reduction type at each prime above $p$. In each case, we found a suitable elliptic curve using $m = 1$ or $2$.
\end{proof}

\subsection{An example with large class group}


We consider the quadratic field $K=\Q(\sqrt{d})$, with $d = 47 \cdot 67 \cdot 101$. The class group of $K$ is $\Z / 122 \Z$.

\begin{proposition} Let $K = \Q(\sqrt{d})$ with $d = 47 \cdot 67 \cdot 101$. There exists an elliptic curve $E / K$ which admits a $K$-rational $p$-isogeny and is semistable at all primes of $K$ above $p$ if and only if $p \in \{2,3,5,7,11,13,19,37\}$ 
\end{proposition}

Although the class group of $K$ is large, our quadratic field is now fixed, and we know the splitting behaviour of each auxiliary prime $q$ in $K$, as well as the order of any $\mathfrak{q} \mid q$ in the class group of $K$.

\begin{proof} We start by following Step 1 of the strategy described at the start of this section with the auxiliary primes $3 \leq q \leq 20$, followed by $q=2$. We were able to show that if $p \geq 23 $ with $p \neq 37$, then the isogeny signature of $(E,V_p)$ cannot be constant.

We then proceed with Step 2 and compute $\mathrm{Norm}(\epsilon^{12}-1)$, for the fundamental unit $\epsilon = 13535 + 24 \sqrt{d}$. Although this has several large prime factors, it in fact has no prime factors $\geq 23$ that split in $K$. 

Next, we continue with Step 3. We start by eliminating the prime $p = 17$ like in the proof of Lemma \ref{lem2}. It remains to eliminate the primes that ramify in $K$, namely $ p = 47, 67,$ and $101$. We work directly with the corresponding modular curves $X_0(p)$. By Lemma \ref{exceptional}, each gives rise to a point $x \in X_0(p)(K)$ satisfying $w_p(x) = x^\tau$. By \citep[p.~337]{Boxquad}, the modular curve $X_0(67)$ has a single non-cuspidal rational point, and no points defined over real quadratic fields. Applying the arguments of Corollary \ref{ratcorol} (see also the proof of Lemma \ref{lem2}), the pair $(E,V_p)$ will not give rise to the non-cuspidal rational point, and so we may eliminate $p=67$. Next, we consider $p = 47$. The curve $X_0(47)$ is hyperelliptic and the Atkin--Lehner involution coincides with the hyperelliptic involution. We therefore obtain a rational point on the quadratic twist of $X_0(47)$ by $d$. However, this twisted curve has no points over $\Q_{101}$, and so we obtain a contradiction. Finally, if $p = 101$, then we would obtain a rational point on the twisted modular curve $X^{(d)}_0(101)$ (see \citep[pp.~323--324]{ozmantwist}). In this case we may apply \citep[Theorem~1.1~Part~(5)]{ozmantwist} to obtain a contradiction. To see this, we start by writing $M = \Q(\sqrt{-101})$. The prime $67$ ramifies in $K$, is unramified in $M$, and each prime of $M$ above $67$ is not principal (and therefore not totally split in the Hilbert class field of $M$). This proves that $X_0^{(d)}(101)(\Q_{67}) = \emptyset$.

For the converse, it suffices to write down suitable elliptic curves for $p = 11$ and $p = 19$. We argue exactly as in the proof of Theorem \ref{Thm2} (for $d > 0$) above. The only difference is that we were unable to compute the Mordell--Weil group of $X_0(p)(K)$ directly using \texttt{Magma}. Instead, we find suitable points by first working with the quadratic twist of $X_0(p)$ by $d$.
\end{proof}


\subsection{An imaginary quadratic field}


We consider the imaginary quadratic field $K = \Q(\sqrt{-5})$ which has class number $2$. We note that we cannot obtain a finite list of possible primes in the case that $K$ is imaginary quadratic of class number $1$. This is because, in this case, if the isogeny signature of $(E,V_p)$ is non-constant, then $R_q = 0$ for all primes $q$. This is unsurprising, since if $K$ is any number field that contains the Hilbert class field of an imaginary quadratic field then there are infinitely many primes for which there exist curves which admit $K$-rational $p$-isogenies (see \citep[p.~2]{Banwait} for more details on this).

\begin{proof}[Proof of Theorem \ref{Thm2} for $d = -5$]

If the isogeny signature of $(E,V_p)$ is constant then we may use the proof of Theorem \ref{Thm1} to eliminate all primes $p \geq 23$ with $p \neq 37$. We will therefore focus on the case that the isogeny signature of $(E,V_p)$ is non-constant and proceed with Step 2. As usual, by interchanging $(E,V_p)$ with $(E / V_p,E[p]/V_p)$ if necessary, we may assume that the isogeny signature of $(E,V_p)$ is $(12,0)$. We use the auxiliary primes $3$ and $7$. 
We have that for $p \geq 17$, \[ p \mid  \gcd(R_3 M_3, R_7 M_7), \] and this tells us that $p \in \{17, 43, 71 \}$. The prime $71$ does not split in $K$, so we may eliminate it. Next, we proceed with Step 3 and eliminate the primes $11$, $17$, and $19$ as in the proof of Lemma \ref{lem2}.

For the converse, we must exhibit an appropriate elliptic curve when $p = 43$. The curve $X_0(43)$ is of genus $3$. We start by searching for rational points on the elliptic curve $X_0^+(43)$, and pull them back to try and find a point $x \in X_0(43)(K) \backslash X_0(43)(\Q)$. We were able to do and then using \texttt{Magma} we wrote down a representative elliptic curve $E$ defined over $K$. We found that this curve was \emph{not} semistable at $\mathfrak{p}_0 \mid 43$. However, the quadratic twist of $E$ by a certain element of valuation $1$ at $\mathfrak{p}_0$ has good reduction at both primes of $K$ above $43$, and this twisted elliptic curve will still have a $K$-rational $43$-isogeny.
\end{proof}

\bibliographystyle{plainnat}

\Addresses

\end{document}